\documentclass[10pt]{amsart}

\usepackage{amsmath,amssymb,amsthm}

\usepackage[utf8]{inputenc}

\usepackage[english]{babel}

\numberwithin{equation}{section}

\usepackage[left=4cm,right=4cm,bottom=3cm,top=3cm]{geometry}

\newcommand{\R}{\mathbb{R}}
\newcommand{\C}{\mathbb{C}}
\newcommand{\N}{\mathbb{N}}
\newcommand{\Z}{\mathbb{Z}}
\newcommand{\abs}[1]{\mathord{\left|#1\right|}}

\newcommand{\eps}{\varepsilon}

\newcommand{\zsum}{\sideset{}{^*}\sum}
\newcommand{\floor}[1]{\left\lfloor #1 \right\rfloor}
\newcommand{\nrm}[1]{\mathord{\left\lVert #1 \right\rVert}}
\newcommand{\eand}{\,\,\,\text{ and }\,\,\,}
\DeclareMathOperator{\dist}{dist}
\DeclareMathOperator{\li}{li}

\newtheorem{prop}{Proposition}
\newtheorem{thm}{Theorem}
\newtheorem{lemma}{Lemma}

\newtheorem{rem}{Remark}

\newcommand{\nce}{\eta_{c,\eps}}

\title[Bounding $\psi(x)$ analytically]{An analytic method for bounding $\psi(x)$.}
\author{Jan Büthe}
\address{Hausdorff Center for Mathematics, Endenicher Allee 62, 53115 Bonn}
\email{jan.buethe@hcm.uni-bonn.de}

\date{\today}

\begin{document}
\begin{abstract}
In this paper we present an analytic algorithm which calculates almost sharp bounds for the normalized remainder term $(t-\psi(t))/\sqrt t$ for $t\leq x$ in expected run time $O(x^{1/2+\eps})$ for every $\eps>0$. The method has been implemented and used to calculate such bounds for $t\leq 10^{19}$. In particular, these imply that $\li(x)-\pi(x)$ is positive for $2\leq x\leq 10^{19}$.
\end{abstract}

\maketitle

\section{Introduction and statement of results}
This paper concerns the problem of calculating limited range approximations to the Chebyshov function
\[
\psi(x) = \sum_{p^m\leq x} \log(p) = x + o(x).
\]
More precisely, we are interested in calculating almost sharp bounds for the normalized error term
\begin{equation}
R_\psi(t) = \frac{t-\psi(t)}{\sqrt t}
\end{equation}
in the prime number theorem for $\psi(t)$ in finite intervals $[x,Lx]$. So far, such calculations seem to have been based on tabulating prime numbers (see e.g. \cite{PT14, Rosser41, RS62}), whereas bounds for unlimited ranges are usually derived analytically (see e.g. \cite{Dusart98,FK15,Rosser41,Rosser75, RS62, Schoenfeld76}). The elementary approach leads to a run time of $\tilde O(x)$ for fixed $L$ and $x\to \infty$, where $f(x) = \tilde O(g(x))$ means \textit{there exists an $A$ such that $f(x) = O(g(x)\log(x)^A)$}. In this paper we present an analytic algorithm for this task, which satisfies the assertion of the following theorem.

\begin{thm}\label{t:main-theorem}
For every triple $(L,\delta,\theta)\in(1,\infty)\times (0,\infty) \times (0,1/2]$, there exist effectively computable constants $C_1=C_1(L,\theta,\delta)$ and $C_2=C_2(L,\theta,\delta)$ and an algorithm which takes $x\geq 2$ and the zeros $\rho$ of the Riemann zeta function with $0< \Im(\rho) \leq C_1 x^{\theta}\sqrt{\log x}$ with an accuracy of $x^{-C_2}$ as input, performs $\tilde O([1+N_e(C_1 x^{\theta}\sqrt{\log x})] x^\theta)$ arithmetic operations on $\tilde O(x^\theta)$ variables of size $C_2 \log x$, where $N_e(T)$ denotes the number of zeros with $0< \Im(\rho) \leq T$ violating the Riemann Hypothesis, and outputs numbers $M_L^+(x)$  and $M_L^-(x)$ satisfying
\begin{equation}\label{e:analytic-bounds}
 \sup_{x\leq t\leq Lx}\pm \frac{t-\psi(t)}{\sqrt t} \leq \pm M_L^\pm(x) \leq  \sup_{x\leq t\leq L x} \pm\frac{t-\psi(t)}{\sqrt{t}} + \delta x^{1/2-\theta},
\end{equation}
where pluses and minuses are to be taken correspondingly.
\end{thm}

If $L$ is sufficiently large the Riemann Hypothesis (RH) implies (\cite[15.2.1, Ex. 2b)]{Montgomery2006})
\begin{equation}
 \liminf_{x\to\infty} \sup_{x\leq t\leq Lx}\pm \frac{t-\psi(t)}{\sqrt t} > 0.
\end{equation}
Furthermore, the expected run time for calculating zeros of the zeta function with imaginary part up to $T$ within an accuracy of $T^{-c}$ for any $c>0$ is $O(T^{1+\eps})$ for every $\eps>0$, assuming RH and simplicity of the zeros (see \cite{OS88}). So if we take $\theta = 1/2$ and $\delta$ sufficiently small, the algorithm calculates almost sharp bounds for $R_\psi(x)$ in $[x,Lx]$ in expected run time $O(x^{1/2+\eps})$ for every $\eps>0$.

The algorithm has been implemented and used to calculate analytic bounds for $x\leq 10^{19}$, using the zeros with imaginary part up to $10^{11}$, whose calculation has been reported in \cite{FKBJ}. The calculated bounds also give rise to improved bounds for the functions
\begin{equation}
\pi(x) = \sum_{p\leq x} 1,\quad \vartheta(x) = \sum_{p\leq x} \log(p),\quad\text{and}\quad \pi^*(x) = \sum_{k\geq 1} \frac{\pi(x^{1/k})}{k}.
\end{equation}
The numerical results are summarized in the following theorem.

\begin{thm}\label{t:explicit-bounds}
The following estimates hold:
\begin{align}
\abs{x-\psi(x)} &\leq 0.94 \sqrt{x} &\text{for $11 < x \leq 10^{19}$,} \label{e:psi-bound-final} \\
x-\vartheta(x) &\leq 1.95 \sqrt x &\text{for $1423 \leq x \leq 10^{19}$,} \label{e:theta-upper}\\
x - \vartheta(x) &> 0.05\sqrt x&\text{for $1\leq x \leq 10^{19}$,} \label{e:theta-lower} \\
\abs{\li(x)-\pi^*(x)}&< \frac{\sqrt{x}}{\log x}&\text{for $2 \leq x \leq 10^{19}$,} \label{e:pi-star-bound} \\
\li(x)-\pi(x) &\leq \frac{\sqrt x}{\log(x)}\Bigl( 1.95 + \frac{3.9}{\log x} + \frac{19.5}{\log(x)^2} \Bigl)&\text{for $2 \leq x \leq 10^{19}$,} \label{e:pi-upper} \\
\li(x) - \pi(x) &> 0&\text{for $2 \leq x \leq 10^{19}$.}\label{e:pi-lower}
\end{align}
\end{thm}
In particular, this gives a new lower bound for the Skewes number, the number $x_s\in [2,\infty)$ where the first sign change of $\li(x)-\pi(x)$ occurs. The last published lower bound appears to be $x_s \geq 1.2\times 10^{17}$ in \cite{PT14}. Furthermore, in the earlier paper \cite{SD2011} the second author claims to have verified $x_s\geq 10^{18}$ but no further explanation is given. In total, the calculations took about $1,200$ hours on a 2.27 GHz Intel Xeon X7560 CPU.

\section*{Acknowledgment}
The author wishes to thank Jens Franke, to whom he owes the idea of applying methods for multiple evaluations of trigonometric sums to this problem and who suggested this topic as part of the author's PhD thesis.

Furthermore, he wishes to thank the anonymous referees for their comments. He is particularly grateful to the referee who spotted a mistake in the proof of Lemma 1, which has been corrected in the final version.

\section{Notations}
In addition to the usual Landau $O$ and Vinogradov $\ll$ notation, we frequently use Turing's big theta notation: $g(t) = \Theta(f(t))$ for $t\in U$ $\Leftrightarrow$ $g(t)\leq \abs{f(t)}$ for $t\in U$. Furthermore, the notation $f(t) \asymp g(t)$ is used for \textit{$f(t) \ll g(t)$ and $g(t) \ll f(t)$}. Finally, $f_\pm(t) := \lim_{h\searrow 0} f(t \pm h)$ denotes the limit from the right, respectively left.

\section{Description of the method}
The basic idea of the method presented in this paper is to use an explicit formula to bound $\psi(t)$ at sufficiently many well-distributed points in $[x,Lx]$ which are then extended to the whole interval by interpolation.

To illustrate the first task recall the well-known approximate version of the von Mangoldt explicit formula
\begin{equation}\label{e:vMangoldt-approx}
\psi(x) = x - \sum_{\abs{\Im(\rho)}<T} \frac{x^\rho}{\rho} - \log 2\pi - \frac 12 \log(1-x^{-2}) + O\Bigl( \frac{x}{T}\log(x)^2\Bigr)
\end{equation}
for $T \ll x$, where the sum is taken over the non-trivial zeros of the Riemann zeta function. Using this formula, approximating $R_\psi(t)$ within an accuracy of $O(x^{1/2-\theta})$ can be done by calculating the contribution of zeros with imaginary part up to $T = C  x^{\theta} \log(x)^2$. Extending these bounds to $[x,Lx]$ with an error of size $O(x^{1/2-\theta})$ can be achieved by calculating approximations at $O(x^\theta)$ well-distributed points in $[x,Lx]$, since if $\xi\in[x,Lx]$ and $0 < y \ll x^{1-\theta}$, then
\begin{align*}
R_\psi(\xi+y) - R_\psi(\xi) &= \frac{\xi+y -(1+O(x^{-\theta})) \xi}{\sqrt{\xi+y}} - \frac{\psi(\xi+y) - (1+O(x^{-\theta}))\psi(\xi)}{\sqrt{\xi+y}} \\
&= O(x^{1/2-\theta}) + O(x^{-\theta - 1/2}) \psi(\xi) + O(x^{-1/2}) (\psi(\xi+y) - \psi(\xi)) \\
&= O(x^{1/2-\theta})
\end{align*}
by the prime number theorem and the Brun-Titchmarsh inequality.

If the approximations are caltulated by directly evaluating the sum over zeros in \eqref{e:vMangoldt-approx} this leads to a run time of $\tilde O(x^{2\theta})$, which outperforms the naive method only for $\theta < 1/2$. This can be improved by using techniques for multiple evaluations of trigonometric sums, such as the Odlyzko-Sch\"onhage algorithm. These allow one to evaluate the contribution of zeros on the critical line to the sum in \eqref{e:vMangoldt-approx} on geometric progressions of length $T$ in run time $\tilde O(T) = \tilde O(x^\theta)$, reducing the run time of the algorithm to $\tilde{O}(x^\theta)$.

In principle one could use an explicit version of \eqref{e:vMangoldt-approx}, but we rather use a continuous approximation to $\psi(x)$ from \cite{BuetheB} for which a similar explicit formula exists. This decreases the truncation bound in the sum over zeros to $T=C' x^\theta \sqrt{\log x}$, saving a factor $\log(x)^{3/2}$. Also, we rather use a simpler FFT method from \cite{FKBJ} in place of the Odlyzko-Sch\"onhage algorithm for multiple evaluation of trigonometric sums.

\subsection{Bounding $\psi(x)$ analytically}
Let
\[
\psi_0(x) = \frac{1}{2}\sum_{p^m<x} \log(p) + \frac{1}{2}\sum_{p^m\leq x} \log(p)
\]
denote the normalized Chebyshov function. We intend to bound $\psi(x)$ in terms of the modified Chebyshov function
\begin{equation}\label{e:psi-psice-equal}
 \psi_{c,\eps}(x)= \psi_0(x) + \sum_{e^{-\eps}x < p^m < e^{\eps}x} \frac{1}{m}M_{x,c,\eps}(p^m)
\end{equation}
introduced in \cite{BuetheB}, where
\begin{multline}\label{e:M-def}
M_{x,c,\eps} (t) = \frac{\log t}{\lambda_{c,\eps}} \Bigl[\chi^*_{[x,\exp(\eps) x]}(t) \int_{-\eps}^{\log(t/x)}\eta_{c,\eps}(\tau) e^{-\tau/2}\, d\tau \\
- \chi^*_{[\exp(-\eps)x,x]}(t) \int_{\log(t/x)}^\eps \eta_{c,\eps}(\tau) e^{-\tau/2}\, d\tau,
\Bigr],
\end{multline}
$\chi_A^*$ denoting the normalized characteristic function which takes the value $1/2$ on the boundary of $A$,
\begin{equation}
\eta_{c,\eps}(\tau) = \frac{c}{\eps \sinh(c)} I_0(c\sqrt{1-(\tau/\eps)^2}),
\end{equation}
$I_0(y)= \sum_{n=0}^\infty (y/2)^{2n}/(n!)^2$ denoting the 0-th modified Bessel function of the first kind, and
$\lambda_{c,\eps} = \int_{-\eps}^\eps \nce(\tau) e^{\tau/2}\, d\tau$.

The function $\psi_{c,\eps}(x)$ is a continuous approximation to $\psi(x)$ and we review some of its properties. The first result provides bounds for $\psi(x)$ in terms of $\psi_{c,\eps}$.

\begin{prop}[{\cite[Proposition 4]{BuetheB}}]\label{p:Mxceps-bound}
Let
\[
\mu_{c}(t) =
\begin{cases}
-\int_{-\infty}^t\eta_{c,1}(\tau)\,d\tau	& t < 0\\
-\mu_{c}(-t)			&t>0\\
 0					&t=0
\end{cases}
\]
and let
\[
\nu_{c}(t) = \int_{-\infty}^t \mu_{c}(\tau)\, d\tau.
\]
Furthermore, let $0\leq \alpha < 1$, $x>100$, and let $0<\eps<10^{-2}$, such that
\[
B = \frac{\eps x e^{-\eps} \abs{\nu_c(\alpha)}}{2(\mu_c)_+(\alpha)} > 1
\]
holds. We define
\[
 A(x,c,\eps,\alpha) = e^{2\eps} \log (e^{\eps}x) \Bigl[ \frac{2 \eps\, x \,\abs{\nu_c(\alpha)}}{\log B} + 2.01 \eps\sqrt{x} +\frac 12 \log\log(2 x^2) \Bigr].
\]
Then we have
\[
\psi(e^{-\alpha\eps}x) \leq \psi_{c,\eps}(x) + A(x,c,\eps,\alpha),
\]
and
\[
\psi(e^{\alpha\eps}x) \geq \psi_{c,\eps}(x) - A(x,c,\eps,\alpha).
\]
\end{prop}

The modified Chebyshov function satisfies a similar explicit formula as $\psi(x)$ but the sum over zeros
converges absolutely and is therefore more accessible to numerical calculations using a subset of the zeros
of $\zeta(s)$. For the purpose of this paper, the following approximate version will suffice.

\begin{prop}\label{p:explicit-formula}
Let $x\geq 10,$ $0<\eps\leq 10^{-4}$ and let
\begin{equation}
\ell_{c,\eps}(t) = \frac{c}{\sinh c} \frac{\sinh(\sqrt{c^2-(\eps t)^2})}{\sqrt{c^2-(\eps t)^2}}
\end{equation}
denote Logan's function \cite{Logan88}. Then we have
\begin{equation}
x-\psi_{c,\eps}(x) = \frac{1}{\ell_{c,\eps}(i/2)}\sum_\rho \frac{\ell_{c,\eps}(\frac\rho i - \frac 1{2i})}{\rho} x^{\rho} + \Theta(2).
\end{equation}
\end{prop}
\begin{proof}
This is a corollary of \cite[Proposition 2]{BuetheB}: since $\ell_{c,\eps}(-z) = \ell_{c,\eps}(z)$ we get
\begin{equation*}
\abs{\sum_\rho \frac{\ell_{c,\eps}(\frac\rho i - \frac 1{2i})}{\rho}} = \frac 12 \abs{\sum_\rho \frac{\ell_{c,\eps}(\frac\rho i - \frac 1{2i})}{\rho(1-\rho)}} \leq   \frac{ \ell_{c,\eps}(i/2)}{2} \sum_\rho \Im(\rho)^{-2} < 0.025 \, \ell_{c,\eps}(i/2)
\end{equation*}
using the bijection $\rho\mapsto 1-\rho$ of non-trivial zeros and \cite[Lemma 17]{Rosser41}. Furthermore, we have $\gamma/2 + 1 + \log(\pi)/2 \leq 1.87$, $-\log(1-x^{-2})/2 \leq 0.006$ and $8\eps\abs{\log\eps}\leq 0.008$, so the assertion follows.
\end{proof}

We have the following tail bounds for truncating the sum over zeros.

\begin{prop}[{\cite[Proposition 3]{BuetheB}}]\label{p:psi-zsum-remainder}
Let $x>1$, $0<\eps\leq 10^{-3}$ and $c\geq 3$. Then we have
\begin{equation}\label{e:rem1}
\sum_{\abs{\Im(\rho)} > \frac c\eps} \abs{ \frac{\ell_{c,\eps}(\frac\rho i - \frac 1{2i}) \, x^{\rho}}{\ell_{c,\eps}(i/2)\, \rho}} \leq 0.16 \frac{x+1}{\sinh(c)} e^{0.71\sqrt{c\eps}} \log(3c) \log\Bigl(\frac c\eps\Bigr).
\end{equation}

Furthermore, if $a\in(0,1)$ such that $a\frac{c}{\eps}\geq 10^{3}$ holds, and if the Riemann Hypothesis holds for all zeros with imaginary part in $(\frac{a c}\eps,\frac c\eps]$, then we have
\begin{equation}\label{e:rem2}
\sum_{\frac{ac}\eps <  \abs{\Im(\rho)} \leq \frac c\eps} \abs{ \frac{\ell_{c,\eps}(\frac\rho i - \frac 1{2i}) \, x^{\rho}}{\ell_{c,\eps}(i/2)\, \rho} } \leq\frac{1+11c\eps}{\pi c a^2} \log\Bigl(\frac c\eps\Bigr)\frac{\cosh(c\sqrt{1-a^2})}{\sinh(c)} \sqrt{x}.
\end{equation}
\end{prop}

\begin{rem}
It should be demonstrated that it is indeed more efficient to approximate $\psi(x)$ this way. Calculating
$\psi(x)$ within an accuracy of $O(x^\delta)$ via the modified Chebyshov function can be done by choosing
\begin{equation*}
\eps = x^{\delta-1} \log(x)^{1/2} \quad\quad\text{and}\quad\quad c = (1-\delta)\log(x)  + 2 \log\log(x) .
\end{equation*}
Since
\[
\abs{\nu_c(0)} \sim \frac{1}{\sqrt{2\pi c}}
\]
for $c\to \infty$ (see \cite[Proposition 5]{BuetheB}), Proposition \ref{p:Mxceps-bound} gives
\[
\psi(x) - \psi_{c,\eps}(x) \ll x^\delta
\]
and from Propositions 2 and 3 we get
\[
\psi_{c,\eps}(x) -  x = \frac{1}{\ell_{c,\eps}(i/2)} \sum_{\abs{\Im(\rho)}< T} \frac{\ell_{c,\eps}(\frac\rho i - \frac 1{2i})}{\rho} x^{\rho} + O(x^{\delta})
\]
with $T=c/\eps \sim (1-\delta) x^{1-\delta}\sqrt{\log x}$. If the same zeros are used in the von Mangoldt explicit formula,
the standard estimate \eqref{e:vMangoldt-approx} gives an error term which is larger by a factor of size $\gg \log(x)^{3/2}$.

\end{rem}


\subsection{Interpolating bounds for $\psi(x)$}
Next, we give an estimate for the interpolation error. For simplicity, we assume that $t-\psi(t)$ changes sign in $[x,Lx]$, or to be more precise: we assume the upper, respectively, lower bound for $R_\psi(t)$ to be positive, respectively negative. This is implied by RH if $L$ is sufficiently large and has been the case in all practical applications.

\begin{prop}\label{p:psi-grid}
Let $10^{9}\leq a < b$ and let
\[
 a=x_0 < x_1 < \dots < x_n = b
\]
be a dissection of $[a,b]$, whose maximal step size
\[
 \Delta = \max\{ x_k - x_{k-1}\mid k=1,\dots,n\}
\]
satisfies $10\leq\Delta\leq10^{-5} a$. Then the following assertions hold:

\begin{enumerate}
\item Let $M>0$ satisfy
\[
 \frac{x_k - \psi(x_k)}{\sqrt{x_k}} \leq M
\]
for $k=0,1,\dots, n$. Then
\[
 \frac{y-\psi(y)}{\sqrt{y}} \leq  1.001 \left[M +\frac{\log a}{\sqrt{a}}\left(\frac{\Delta}{\log\Delta} + \log\log(a^2)\right)\right]
\]
holds for all $y\in [a,b]$.

\item Let $m<0$ satisfy
\[
  \frac{x_k - \psi(x_k)}{\sqrt{x_k}} \geq m
\]
for $k=0,1,\dots, n$. Then 
\[
\frac{y-\psi(y)}{\sqrt{y}} \geq 1.001\left[ m - \frac{\log a}{\sqrt{a}}\left(\frac{\Delta}{\log\Delta} + \log\log(a^2)\right) \right]
\]
holds for all $y\in [a,b]$.
\end{enumerate}
\end{prop}

\begin{proof}
We start by proving
\begin{equation}\label{e:psi-bound}
\psi(x)-\psi(x-y) \leq \log (x) \left(1.0001\frac{\Delta}{\log\Delta} + \log\log (x^2)\right)
\end{equation}
for $x\geq a \geq 10^9$ and $0\leq y \leq \frac{\Delta}{2}$. Since $\frac{\Delta}{\log\Delta} \geq \frac{10}{\log 10} > 4$ we may assume $y\geq 3$. The Brun-Titchmarsh inequality, as stated in \cite{MV1973}, and the trivial estimate
\[
\#\{p \mid p^m\in[X-Y,X]\} \leq 2 \frac{Y}{m}X^{1/m} + 1,
\]
which holds for $0<2Y<X$, yield
\begin{align*}
\psi(x)-\psi(x-y)  &\leq \log(x) \sum_{x-y\leq p^m\leq x} \frac{1}{m} \\
&\leq \log(x)\left(\frac{2y}{\log y} + \sum_{m=2}^{\floor{2\log x}} \Bigl( 2\frac{y}{m^2} x^{1/m-1} + \frac{1}{m}\Bigr)\right).
\end{align*} 
Since
\begin{equation*}
\sum_{m=2}^{\floor{2\log x}} \frac 1m \leq \int_{1}^{2\log x} \frac{dt}{t} \leq \log\log(x^2)
\end{equation*}
and
\begin{align*}
\sum_{m=2}^{\floor{2\log x}} 2\frac{y}{m^2} x^{1/m-1} 
	&\leq \frac{y}{2\sqrt x} + \frac{2y}{x^{2/3}}\int_2^\infty \frac{dt}{t^2} \\
	&\leq \frac{y}{\sqrt x}\Bigl(\frac 12 + x^{-1/6}\Bigr) < 0.6 \frac{y}{\sqrt x} \leq 0.0002 \frac{y}{\log y},
\end{align*}
this implies \eqref{e:psi-bound} since $y\mapsto \frac{y}{\log y}$ increases monotonically for $y>e$.

Now let $x\in\{x_k\}_{k=1}^n$ and let $0\leq y \leq \Delta/2$.  Then we have
\begin{equation}\label{e:psi-minus}
\frac{x-y-\psi(x-y)}{\sqrt{x-y}} = \frac{x-\psi(x)}{\sqrt{x-y}} + \frac{\psi(x)-\psi(x-y)}{\sqrt{x-y}} - \frac{y}{\sqrt{x-y}}.
\end{equation}
Now if $m<0$ and $M>0$ satisfy the conditions in the theorem, then
\begin{equation*}
1.001 m \leq \frac{\sqrt x}{\sqrt{x-y}} m \leq \frac{x-\psi(x)}{\sqrt{x-y}} \leq \frac{\sqrt x}{\sqrt{x-y}} M \leq 1.001 M.
\end{equation*}
Furthermore \eqref{e:psi-bound} gives
\begin{align*}
0\leq  \frac{\psi(x)-\psi(x-y)}{\sqrt{x-y}}
	&\leq \frac{\log  x}{\sqrt{x-y}}\left(1.0001\frac{\Delta}{\log \Delta} + \log\log(x^2)\right)  \\
	&\leq 1.001 \frac{\log a}{\sqrt a}\left(\frac{\Delta}{\log \Delta} + \log\log(a^2)\right).
\end{align*}
Since
\[
0\leq \frac{y}{\sqrt{x-y}} \leq 1.0001 \frac{\Delta}{\sqrt a}	\leq 1.001 \frac{\log(a)\Delta}{\sqrt{a}	\log\Delta}
\]
the bound \eqref{e:psi-minus} yields
\begin{align*}
1.001\left(m - \frac{\log a}{\sqrt a}\Bigl(\frac{\Delta}{\log\Delta} + \log\log (a^2)\Bigr)\right) 
	&\leq \frac{x-y-\psi(x-y)}{\sqrt{x-y}} \\
	&\leq 1.001\left(M + \frac{\log a}{\sqrt a}\Bigl(\frac{\Delta}{\log\Delta} + \log\log (a^2)\Bigr)\right).
\end{align*}
The estimates
\begin{align*}
1.001\left(m - \frac{\log a}{\sqrt a}\Bigl(\frac{\Delta}{\log\Delta} + \log\log (a^2)\Bigr)\right) 
 &\leq \frac{x+y-\psi(x+y)}{\sqrt{x+y}}  \\
&\leq 1.001\left(M + \frac{\log a}{\sqrt a}\Bigl(\frac{\Delta}{\log\Delta} + \log\log (a^2)\Bigr)\right)
\end{align*}
for $x\in\{x_k\}_{k=0}^{n-1}$ and $0\leq y \leq \Delta/2$ are proven in an analogous way.\qedhere

\end{proof}

\subsection{Fundamental Theorem}
We can now state the fundamental theorem for the analytic method, which reduces the problem of bounding $R_\psi(t)$ on $[x,Lx]$ to efficiently approximating $\psi_{c,\eps}(t)$ at finitely many points. This is then dealt with in the next section.

\begin{thm}\label{t:psi-bounds}
Let $0< \eps < 10^{-4}$ and let $0\leq \alpha \leq 1$ satisfy
\[
\frac{\eps \,x\, \nu_c(\alpha)}{2 (\mu_c)_+(\alpha)} > 10.
\]
Furthermore, let $e^{\alpha\eps}10^{9}\leq a < b$ and let
\[
 a=x_0 < x_1 < \dots < x_n = b
\]
be a dissection of $[a,b]$ whose maximal step size
\[
 \Delta = \max\lbrace \abs{x_k - x_{k-1}}\mid k=1,\dots,n\rbrace
\]
satisfies $10\leq \Delta \leq 10^{-5} a$.

We define the error terms
\begin{align*}
 \mathcal E_1 &= 1.001 \, \alpha\,\eps \sqrt{b},	\\
\mathcal E_2 &= 2.02\log(b)\left( \eps\,\sqrt{b}\,\abs{\nu_c(\alpha)}\log\Bigl(\frac{\eps \, b \, \abs{\nu_c(\alpha)}}{2(\mu_c)_+(\alpha)}\Bigr)^{-1} + \eps + \frac{\log\log(2a^2)}{4\sqrt a}\right)	\\
\intertext{and}
\mathcal E_3 &= 1.001\frac{\log a}{\sqrt a}\left(\frac{\Delta}{\log\Delta} + \log\log(a^2)\right).
\end{align*}

Then the following assertions hold:

\begin{enumerate}
\item Let $M>0$ satisfy
\[
 \frac{x_k - \psi_{c,\eps}(x_k)}{\sqrt{x_k}} \leq M
\]
for $k=0,1,\dots, n$. Then
\[
 \frac{y-\psi(y)}{\sqrt{y}} \leq  1.01 \left(M + \mathcal E_1 +\mathcal E_2 + \mathcal E_3\right)
\]
holds for all $y\in  [e^{\alpha\eps} a, b]$.
\item Let $m<0$ satisfy
\[
  \frac{x_k - \psi_{c,\eps}(x_k)}{\sqrt{x_k}} \geq m
\]
for $k=0,1,\dots, n$. Then
\[
\frac{y-\psi(y)}{\sqrt{y}}   \geq 1.01  \left(m - \mathcal E_1 - \mathcal E_2 - \mathcal E_3\right)
\]
holds for all $y\in [a,e^{-\alpha\eps} b]$.
\end{enumerate}
\end{thm}

\begin{proof}
We start with the proof of the first assertion concerning the upper bound. Let $\tilde x = e^{\alpha\eps} x$. Then Proposition \ref{p:Mxceps-bound} yields
\begin{align*}
 \frac{\tilde x_k - \psi(\tilde x_k)}{\sqrt{\tilde x}} &\leq \frac{\tilde x_k - x_k}{\sqrt{\tilde x_k}} + \frac{x_k-\psi_{c,\eps}(x_k)}{\sqrt{\tilde x_k}} + \frac{A(x_k,c,\eps,\alpha)}{\sqrt{\tilde x_k}} \\
&= 2\sinh(\alpha\eps /2) \sqrt{x_k} + e^{-\alpha\eps/2} \frac{x_k-\psi_{c,\eps}(x_k)}{\sqrt{x_k}} + e^{-\alpha\eps/2} \frac{A(x_k,c,\eps,\alpha)}{\sqrt{x_k}}.
\end{align*}
Under the suppositions of the proposition we have
\[
 2\sinh(\alpha\eps /2) \sqrt{x_k} \leq \mathcal E_1,
\]
and
\[
 \frac{A(x_k,c,\eps,\alpha)}{\sqrt{x_k}} \leq \mathcal E_2.
\]
Therefore,
\[
 \frac{\tilde x_k - \psi(\tilde x_k)}{\sqrt{\tilde x_k}} \leq M + \mathcal E_1 + \mathcal E_2
\]
for $k=0,\dots, n$ and Proposition \ref{p:psi-grid} yields the desired estimate
\begin{align*}
 \frac{y - \psi(y)}{\sqrt{y}} 
 &\leq 1.001\left(M + \mathcal E_1 + \mathcal E_2 + \frac{\log(\tilde a)}{\sqrt{\tilde a} }\Bigl(\frac{\tilde \Delta}{\log\tilde\Delta} + \log\log(\tilde a^2)  \Bigr)\right) \\
&\leq 1.01\left(M+ \mathcal E_1+ \mathcal E_2 + \mathcal E_3\right)
\end{align*}
for all $y\in[\tilde a, \tilde b].$

The lower bound estimate follows in an analogous way by using
\[
\frac{\hat x_k-\psi(\hat x_k)}{\sqrt{\hat x_k}}  \geq -2\sinh(\alpha\eps /2) \sqrt{x_k} + e^{\alpha\eps/2} \frac{x_k-\psi_{c,\eps}(x_k)}{\sqrt{x}} - e^{\alpha\eps/2} \frac{A(x_k,c,\eps,\alpha)}{\sqrt{x_k}},
\]
where $\hat x_k = e^{-\alpha\eps} x_k$.\qedhere

\end{proof}

\section{Evaluation of $\psi_{c,\eps}$}
We intend to evaluate the sum over zeros,
\begin{equation}\label{e:psi-trigsum}
\zsum_{\abs{\Im(\rho)}< T} \frac{\ell_{c,\eps}(\rho/i-1/2i)}{\rho} x^{\rho},
\end{equation}
in the explicit formula for $\psi_{c,\eps}$ for many values of $x\in[x_0,L x_0]$. If we take
$y=\log(x)$, denote non-trivial zeros of $\zeta(s)$ by $\rho=\beta  + i\gamma$ with $\beta,\gamma\in \R$, normalize with the factor $e^{-y/2}$, and remove possible violations of the RH, we encounter a trigonometric sum
\begin{equation}\label{e:F_psi}
F_{\psi,T}(y) = \sum_{\substack{\abs{\gamma}<T\\\beta=1/2}} a_\rho e^{i y \gamma},
\end{equation}
where
\[
a_\rho =   \frac{\ell_{c,\eps}(\rho/i-1/2i)}{\rho}.
\]

Such trigonometric sums can be evaluated efficiently on equidistant grids using the Fast Fourier Transform (FFT). In this case
we can calculate $O(T)$ values of $F_{\psi,T}(y)$ using $\tilde{O}(T)$ arithmetic operations on variables of size $O(\log(T))$ (see \cite{OS88}). Furthermore, the Fourier transform of $F$ is supported on $[-T,T]$ so that $F(y)$ can be recovered from samples $F(n\pi/\beta)$ for some $\beta>T$ by bandlimited function interpolation, where a single evaluation can be done in $\tilde{O}(1)$ (see \cite{Odlyzko92}).

\subsection{Multiple evaluations of trigonometric sums}\label{ss:trigs-eval}
Let
\begin{equation}\label{e:trigsum}
 F(y) = \sum_{j=1}^N a_j e^{i\gamma_j y}
\end{equation}
with $\gamma_j\in\R$ and $a_j\in\C$. The first author of \cite{FKBJ} proposed a simple method, based on the FFT to evaluate $F(y)$ simultaneously at integer values $y\in [-Y,Y]\cap \Z$. The method is similar to the Odlyzko-Sch\"onhage algorithm \cite{OS88}.

We briefly restate the algorithm and analyze the run time for the application in mind. The algorithm is based on rounding $e^{i\gamma_j}$ onto the next $R$th root of unity, where $R=2^r$ is a power of $2$. Let $n_j\in \Z$ such that
\[
 \delta_j := \gamma_j - \frac{2\pi n_j}{R} = \Theta\Bigl(  \frac{\pi}{R} \Bigr).
\]
Furthermore, let
\[
 P(t) = b_0 + \dots + b_n t^n
\]
be a polynomial approximating $f(t)=\exp(it\frac{\pi Y}R)$ in $[-1, 1]$. Then we have
\begin{align*}
 F(y) &= \sum_{j=1}^N a_j e^{2\pi i n_j y/R} e^{iy\delta_j}\\
&= \sum_{j=1}^N a_j e^{2\pi i n_j y/R} P(y\delta_j\tfrac R{\pi Y}) + \Theta\Bigl(\nrm{f - P;\,C^0([-1, 1])} \sum_{j=1}^N \abs{a_j}\Bigr)
\end{align*}
for $y\in[-Y,Y]$, where
\[
\nrm{g;\,C^0([a,b])} := \sup_{t\in[a,b]} \abs{g(t)}
\]
denotes the supremum norm on $[a,b]$. Now let
\[
 f_\ell(k) = \sum_{\substack{j=1\\n_j\equiv k \bmod R}}^N a_j\Bigl(\frac{R\delta_j}{\pi Y}\Bigr)^\ell \eand \hat f_\ell(y) = \sum_{k=1}^R f_\ell(k) e^{2\pi i k y/R}.
\]
Then we have
\begin{align*}
 \sum_{j=1}^N a_j e^{2\pi i n_j y/R} P(x\delta_j\tfrac R{\pi Y}) &= \sum_{\ell=1}^n b_\ell\,y^\ell\sum_{j=1}^N a_j \delta_j^\ell e^{2\pi in_j y/R} \\
&= \sum_{l=1}^n b_\ell \hat f_\ell(y) y^\ell
\end{align*}
and all values of $\hat f_\ell$ on $\Z/R \Z$ may be calculated appealing to the FFT. For the polynomial $P$ we choose
the polynomial $P_n$ of degree $\leq n$ which interpolates $f(t)$ at the zeros $\cos(\frac{2k-1}{2n+2}\pi),$ $k=1,2,\dots, n+1$, of the $n+1$th Chebyshov polynomial. The standard error estimate for polynomial interpolation then gives the bound
\begin{equation}\label{e:cheb-error-bound}
\nrm{f - P_n;\,C^0([-1, 1])} \leq \Bigl(\frac{\pi Y}{2 R}\Bigr)^{n+1} \frac{\sqrt{8}}{(n+1)!}.
\end{equation}
We then get the following result.

\begin{prop}\label{p:trigs-evaluation}
Assume in \eqref{e:trigsum} that there exist constants $B,C\geq 0$ such that $\abs{a_j}\leq B j^C$ for all $j$, and that $\gamma_j\in [0,2\pi)$, and let $D,\alpha>0$. Then there exists an $N_0(\alpha,B,C,D)$ such that for all $N,Y\in \N$ satisfying $N>N_0$ and
$\log(N)^{-D} \leq N/Y \leq \log(N)^D$ the algorithm above takes each $a_j,$ $j=1,\dots ,N,$ with an accuracy of $N^{-2\alpha - 3}$ and each $\gamma_j, $ $j=1,\dots, N,$ with an accuracy of $N^{-2\alpha - C - 4}$ as input and calculates $F(y)$ for all 
 $y\in [-Y,Y]\cap \Z$ within an accuracy of $N^{-\alpha}$ performing $\tilde O(N)$ arithmetic operations on $\tilde O(N)$ variables of size $O(\log N)$, where the implied constants depend on $\alpha, B, C$ and $D$ only.
\end{prop}

\begin{proof}
Let $R=2^r$ denote the power of two which is closest to $N$. Then, in view of \eqref{e:cheb-error-bound}, we have
\begin{equation}
F(y) = \sum_{1\leq n \leq \log N} b_\ell \hat f_\ell(y) y^\ell + O(N^{-2\alpha})
\end{equation}
for every $\alpha> 0$. It is easily seen from the discrete orthogonality of the Chebyshov polynomials $T_k$ that $b_\ell \ll 3^n \ll N^2$, since the coefficients of $T_k$ are bounded by $3^k$ in absolute value. Furthermore, we have $\hat f_\ell(y) \ll N^{C+1}/Y^\ell$ and trivially $y^\ell \ll Y^\ell$. It therefore suffices to calculate $b_\ell$ within an accuracy of $O(N^{-2\alpha - C - 1})$, $\hat f_\ell(y)$ within an accuracy of $O(Y^{-\ell}N^{-2\alpha - 2})$, and $y^\ell$ within an accuracy of $O(Y^\ell N^{-2\alpha - C - 3})$ in order to calculate $F(y)$ within an accuracy of $O(N^{-2\alpha})$, which can all be carried out on variables of size $O(\log N)$. The calculation of $b_\ell$ takes $\tilde O(1)$ arithmetic operations. For the calculation of $\hat f_\ell$ it suffices if the input variables $a_j$ and $\gamma_j$ are given within an accuracy of $N^{-2\alpha - 3}$, respectively $N^{-2\alpha - C - 4}$ and all values $\hat{f}_\ell(y)$ are calculated via FFT performing $\tilde{O}(N)$ arithmetic operations on $\tilde O(N)$ variables. Calculating $F(y)$ for a single value of $y$ then takes $\tilde O(1)$ arithmetic operations, so the assertion follows.
\end{proof}

\subsection{Bandlimited function interpolation}
The method outlined in the preceding section is sufficient to obtain an algorithm satisfying Theorem \ref{t:main-theorem}. But for practical applications it can be necessary to reduce the memory requirement of the algorithm by sub-dividing the sum over zeros. Then the number of evaluations is much larger than the number of summands in the trigonometric sum and it is favorable to calculate sufficiently many samples of the trigonometric sum to obtain intermediate values by bandlimited function interpolation instead of repeatedly applying the method from the previous section. 

We recall the interpolation formula from \cite{Odlyzko92} which is a modification of the well-known Shannon-Nyquist-Whittacker interpolation formula and give an explicit estimate for truncating the infinite sum.

\begin{prop}\label{p:bi}
 Let
\[
 F(y) = \sum_{j=1}^N a_j e^{i\gamma_j y},
\]
where $\gamma_j\in\R$ and let $\tau = \max_j\{\abs{\gamma_j}\}$. If $\beta,$ $\lambda$ and $\eps$ satisfy the inequalities
\[
\tau \leq \lambda - \eps < \lambda + \eps \leq \beta,
\]
then we have
\begin{equation}\label{e:bi-formel}
F(y) = \frac{\lambda}{\beta} \sum_{n\in\Z} F\bigl(\frac{\pi n}{\beta}\bigr) \frac{\sin(\lambda(y- \frac{\pi n}\beta))}{\lambda(y-\frac{\pi n}\beta)} \ell_{c,\eps}\bigl(y - \frac{\pi n}\beta).
\end{equation}
Furthermore, if $A = \sum_{j=1}^N \abs{a_j}$, then we have
\begin{multline}\label{e:bi-approx}
\abs{\frac{\lambda}{\beta} \sum_{\abs{y-\frac{\pi n}{\beta}} > \frac{c}{\eps}} F\bigl(\frac{\pi n}{\beta}\bigr) \frac{\sin(\lambda(y-\frac{\pi n}\beta))}{\lambda(y-\frac{\pi n}\beta)} \ell_{c,\eps}\bigl(y - \frac{\pi n}\beta\bigr)} \\
 \leq \frac{2 A}{\sinh(c)}\left( \frac{\log\bigl(e(c+1)\bigr)}\pi + \frac{2\eps}{\beta} \right).
\end{multline}
\end{prop}

\begin{proof}
The proof of \eqref{e:bi-formel} is outlined in \cite{Odlyzko92}, so we only prove the bound \eqref{e:bi-approx}. We start by estimating the contribution of summands with
\[
y-\frac{n\pi}{\beta} > \frac c\eps
\]
to \eqref{e:bi-formel}. Using the bounds
\[
\abs{\frac{\sin(x)}{x}}  \leq \frac{1}{x},\quad \abs{\ell_{c,1}(y)}\leq \frac c{\sinh(c)}\min\Bigl\{1,\frac{1}{\abs{y}-c} \Bigr\}\eand \abs{F(t)} \leq A,
\]
which hold for $x\neq 0$, $y> c$ and $t\in\R$, we get
\begin{align*}
 &\frac{\lambda}{\beta } \sum_{y-\frac{n\pi}{\beta} > \frac c\eps} \abs{ F\bigl(\frac{\pi n}{\beta}\bigr) \frac{\sin(\lambda(y-\frac{\pi n}\beta))}{\lambda(y-\frac{\pi n}\beta)} \ell_{c,\eps}\bigl(y - \frac{\pi n}\beta) } \\
 &\quad\leq \frac{A c}{\beta\sinh(c)}\left[ \frac{\eps}{c}\Bigl(\frac{\beta}{\pi \eps} + 1 \Bigr) + \sum_{y-\frac{n\pi}{\beta} > \frac {c+1}\eps} \frac{1}{y-\frac{n\pi}\beta}\frac{1}{\eps(y-\frac{n\pi}\beta )-c}\right] \\
 & \quad\leq  \frac{A c}{\beta\sinh(c)}\left[ \frac {2\eps} c + \frac{\beta}{\pi c}  + \int_{-\infty}^{\frac\beta\pi(y-\frac{c+1}{\eps})} \frac{dt}{(y-\frac{\pi}\beta t) (\eps(y-\frac{\pi}\beta t)-c)} \right] \\
 & \quad= \frac{A}{\sinh(c)}\left[\frac{2\eps}\beta + \frac 1 \pi \log\bigl(\pi(c+1)\bigr)\right].
\end{align*}

An analogous calculation gives the same estimate for the contribution of summands with $y-\frac{n\pi}{\beta} < -\frac c\eps$.
\end{proof}

\section{Run time analysis}
\subsection{Proof of Theorem \ref{t:main-theorem}}
Let $L>1$, $\delta>0$ and $\theta\in[1/2,1)$. We may assume $x_0$ to be sufficiently large, since the task can always be carried out trivially using the Eratosthenes sieve in finite ranges. For simplicity we focus on proving the assertion concerning the upper bound $M_L^+(x_0)$. The considerations for $M_L^-(x_0)$ are almost the same.

We first address the problem of bounding $\psi(t)$ in $I=[x_0,L x_0]$. Let
\begin{equation}
\tilde \psi_{c,\eps}(x) = x - \zsum_{\abs{\gamma}\leq c/\eps}\frac{\ell_{c,\eps}(\frac\rho i - \frac 1{2i}) \, x^{\rho}}{\ell_{c,\eps}(i/2)\, \rho},
\end{equation}
and let $\eta_1 <  \theta$. If we assume $x_0^{\eta_1-1} < \eps < x_0^{-\eta_1}$ and take $c = \theta\log x_0 + \log\log x_0 + \log\log\log x_0 - \log(\delta/40)$, then Propositions \ref{p:explicit-formula} and \ref{p:psi-zsum-remainder} give the bound
\begin{equation}
\abs{\psi_{c,\eps}(t) - \tilde\psi_{c,\eps}(t)} \leq 2 + e^{o(1)} \frac\delta{40} x_0^{1-\theta} < \frac{\delta}{20} x_0^{1-\theta}
\end{equation}
for $t\in I$ and $x_0$ sufficiently large. Consequently, we may take $\eps = \eta_2 x_0^{-\theta} \sqrt{\log x_0}$ for every $\eta_2>0$, and since $(\mu_c)_+(0) = 1/2$ and $\abs{\nu_c(0)} \sim (2\pi c)^{-1/2}$, we may achieve 
\begin{equation}
\abs{\psi(t) - \psi_{c,\eps}(t)} \leq C(\theta) \eta_2 x_0^{1-\theta} < \frac{\delta}{20}x_0^{1-\theta} 
\end{equation}
for $t\in I$ by use of Proposition \ref{p:Mxceps-bound}. Now assume we may calculate $t-\tilde \psi_{c,\eps}(t)$ for $t\in I$ within an accuracy $< \delta x_0^{1-\theta}/20$ and denote this approximation by $R(t)$. Then we get
\begin{equation}\label{e:num-bounds}
\frac{R(t)}{\sqrt t} - \frac{3\delta}{20} x_0^{1/2-\theta} \leq \frac{t-\psi(t)}{\sqrt{t}} \leq \frac{R(t)}{\sqrt t} + \frac{3\delta}{20} x_0^{1/2-\theta}
\end{equation}
for $t\in I$ which we intend to interpolate. We cannot use Proposition \ref{p:psi-grid} directly since we assumed the bounds to have opposite sign and since this would also give a slightly weaker result where applicable. Instead we estimate trivially, which increases the number of grid points by a factor $\log x_0$. Let $S\subset I$ be a finite subset satisfying $\dist ( \{s\}, S\setminus \{s\}) \leq \eta_3 x_0^{1-\theta}/\log(x_0)$ for all $s\in S\cup\{x_0,2x_0\}$. Now let $s\in S$, $s\pm t\in I$ and $\abs{t} \leq \eta_3 x_0^{1-\theta}/\log(x_0)$, where $\eta_3$ is sufficiently small. Then estimating as in \eqref{e:psi-minus} gives
\begin{equation}\label{e:extension}
\frac{(s\pm t) - \psi(s\pm t)}{\sqrt{s\pm t}} = \frac{s - \psi(s)}{\sqrt{s}}(1+ O(x_0^{-\theta})) + \frac{\delta}{20}x^{1/2-\theta}
\end{equation}
for $x_0$ sufficiently large.

Now let
\begin{equation}
M_0 = \max_{t\in I} \frac{t-\psi(t)}{\sqrt{t}}.
\end{equation}
Then in view of \eqref{e:num-bounds} the approximation $R(s)/\sqrt s$ yields an upper bound $M_1$ satisfying
\begin{equation}
 \frac{s - \psi(s)}{\sqrt{s}} \leq M_1 \leq M_0 + \frac{3\delta}{20}x_0^{1/2-\theta}
\end{equation}
for $s\in S$. By \eqref{e:extension} this extends to the bound
\begin{equation}
 \frac{t - \psi(t)}{\sqrt{t}} \leq M_1 (1+ O(x_0^{-\theta})) + \frac{\delta}{20}x^{1/2-\theta} \leq M_0 + \delta x_0^{1-\theta}
\end{equation}
for $t\in I$, since $M_0 = o(\sqrt{x_0})$ for $x_0\to \infty$.

It remains to analyze the run time for evaluating $R(s)$ on such a set $S$. We may take $S = \{ \exp(y_0 + k h)\mid k\in \Z\}\cap I$, where $y_0 = \log(\sqrt{L} x_0)$ and $h = \eta_4 x_0^{-\theta}/\log(x_0)$. We take $T = c/\eps \sim C(\delta,\theta) x^\theta \sqrt{\log x}$ in \eqref{e:F_psi} and consider the trigonometric sum $F(y) =  F_{\psi,T} (y_0 + yh)$ which we intend to evaluate within an accuracy of $\delta x_0^{-\theta}/40$ for $y\in [-Y,Y]\cap \Z$, where $Y = \max \{ \abs k \mid \exp(y_0 + k h) \in I\}$. If $x_0$ is sufficiently large, then $F(y)$ satisfies the suppositions of Proposition \ref{p:trigs-evaluation} (after reducing $\gamma h$ modulo $2\pi$ and evaluating $a_\rho e^{i\gamma y_0}$ in \eqref{e:psi-trigsum}, which is done in $\tilde O(x^\theta)$) with $B = 1$, $C=0$ and $D=1$. We have $N\asymp(x_0^\theta \log(x_0)^{3/2})$ and $Y \asymp(x_0^\theta \log(x_0))$, so we may evaluate $F(y)$ within an accuracy of $N^{-2\theta}$ using $\tilde O(N) = \tilde O(x^\theta)$ arithmetic operations on $\tilde O(x^\theta)$ variables of size $O(\log x)$, where the implied constants only depend on $L$, $\theta$ and $\delta$. Furthermore, we may evaluate the contribution of a single zero violating the Riemann hypothesis to the explicit formula within sufficient accuracy performing $\tilde O (x^\theta)$ arithmetic operations on variables of size $O(\log x)$. For $x_0$ sufficiently large, this yields the desired accuracy and we can recover the values $R(s)/\sqrt s$ with an error $< \delta x^{1/2-\theta}/20$.\qed

\subsection{Reducing the memory requirement}

One may reduce the space requirement of the algorithm by splitting the sum over zeros, applying the method from section \ref{ss:trigs-eval} to the partial sums and using bandlimited function interpolation to calculate intermediate values. One then does not evaluate the full trigonometric sum anymore but rather calculates upper and lower bounds for the partial sums which are subsequently used to calculate bounds for the full trigonometric sum.

More precisely, let $(L,\delta,\theta)$ be an admissible triple in Theorem \ref{t:main-theorem}. Then we proceed as in the proof of Theorem \ref{t:main-theorem}, but bound $F_{\psi,T}$ in the following way. Let $N=\floor{x^\eta}$, let $\rho_n = 1/2+i\gamma_n$ be an enumeration of the zeros in the upper half plane satisfying $RH$ ordered by increasing absolute value and define
\begin{equation}
F_k(y) = e^{-iy \tau_k} \sum_{kN <n \leq (k+1)N} a_{\rho_n} e^{iy\gamma},
\end{equation}
where $\tau_k = (\gamma_{(k+1)N} - \gamma_{kN+1})/2$. Since $\gamma_n \asymp n / \log n$ the functions $F_k$ have bandwidth $\ll x^\gamma$ and can thus be recovered from samples $F_k(h \ell)$ where $h \gg x^{-\eta}$. In view of Proposition \ref{p:bi} it thus suffices to calculate $O(x^\eta)$ samples which by Proposition \ref{p:trigs-evaluation} can be done performing $\tilde O(x^\eta)$ arithmetic operations on $\tilde O(x^\eta)$ variables of size $O(\log x)$. Now for each $k$ with $\gamma_{(k+1)N} \leq T=  C_1(L,\delta,\theta) x^\theta \sqrt{\log x}$ the required $\tilde O(x^{\theta})$ evaluations can be done in $\tilde O(x^\theta)$ using the interpolation formula from Proposition \ref{p:bi}. For each $k$ only the maximal and minimal values of $ \Re e^{iy \tau_k} F_k(y)$ are stored, from which one recovers upper and lower bounds for $\Re F_{\psi,T}(y)$. There are $\tilde O(x^{\theta-\eta})$ values $k$ to be considered, so in total the algorithm performs $\tilde O(x^{2\theta - \eta})$ arithmetic operations on $\tilde O(x^\eta)$ variables of size $O(\log x)$.

It should be noted that the additional error from splitting the sum over zeros could be avoided by adapting the method from \cite{Hiary11} to this problem. This way one would split both the trigonometric sum and the interval in question and use direct evaluation combined with bandlimited function interpolation on every sub interval. For the calculations reported in this paper this additional error was rather small (less than $1\%$ of the calculated bounds) and the author did not try out this method.

\section{Numerical results}
The algorithm has been implemented for $L=2$, $\theta=1/2$ and variable $\delta$ and used to calculate analytic bounds in the range between $10^{10}$  and $10^{19}$.

Function evaluations have been done using the multi-precision library \textit{MPFR} and the crucial calculations
have been carried out using a 64-bit fixed point arithmetic.

The calculations used the zeros with imaginary part up to $10^{11}$ whose calculation has been reported in \cite{FKBJ} and which
were given within an accuracy of $2^{-64}$. The amount of memory was limited to 340 GB  
which required a sub-division of the sum over zeros for $x\geq 4\times 10^{14}$, the maximal amount of summands being $1.25\times 10^{10}$. For the largest calculation, concerning the interval $[5.12\times 10^{18},1.024\times 10^{19}]$ the sum was divided into 13 pieces. This calculation took $290$ hours on a 2.27 GHz Intel Xeon X7560 CPU. The run time could have been reduced further by parallelizing the interpolation routine, which accounted for half of the computing time. In total, the calculations took less than $1,200$ CPU hours.

The largest value of a partial sum
\begin{equation}
\frac{1}{\ell_{c,\eps}(i/2)}\sum_{\abs{\Im(\rho)}<T} \frac{\ell_{c,\eps}(\frac{\rho}{i}-\frac{1}{2i})}{\rho} x^{\rho-1/2}
\end{equation}
in the explicit formula for the normalized remainder term $(t-\psi_{c,\eps}(t))/\sqrt t$ that occurred in the calculations was
$0.83545670\dots$ at $x=36219716654216.6\dots$ with $c=26$, $\eps = 1.7\times 10^{-8}$ and $T= 917,647,060$ and
the smallest value was $-0.783738372378$ at  $x=1325006525152927089.\,\dots$ with $c=31$, $\eps = 2.5\times 10^{10}$
and $T=3,221,225,472$. The program aims to calculate the sum over zeros within an accuracy of $10^{-10}$. This does not include round-off errors, which could be larger but can still be shown to be bounded by $0.016$ in these calculations \cite{BuetheDiss}. In addition, the extremal values have been counter-checked by direct evaluation of the sums in question and the largest deviation was $<6\times 10^{-12}$. A complete list of parameters and calculated values is given in the appendix to \cite{BuetheDiss}.

The calculated bounds are listed in Table \ref{tb:psi-bounds}. In addition the bounds 
\begin{equation}\label{e:psi-bounds-eratosthenes}
-0.8 \leq R_\psi(t) \leq 0.81
\end{equation}
for $100\leq t \leq 5\times 10^{10}$ have been calculated using the Eratosthenes sieve. Together these
imply the bound \eqref{e:psi-bound-final}, where the validity for $11<t<100$ is easily checked by direct evaluation.

\begin{table}
\begin{center}
\caption{Upper and lower bounds $M_\psi^\pm(x)$ for $\frac{t-\psi(t)}{\sqrt t}$ in $[x,2x]$}\label{tb:psi-bounds}

\vspace*{.5cm}

\begin{tabular}{r|r|r||r|r|r}
$x$	&$ M_\psi^-(x)$		&$ M_\psi^+(x)$		&$x$ 	&$ M_\psi^-(x)$	&	$ M_\psi^+(x)$	\\
\hline
$10^{10}$					&$-.77$	&$.85$		&$10^{12}$				&$-.80$	&$.81$	\\	
$2\times10^{10}$		&$-.75$	&$.64$		&$2\times10^{12}$	&$-.79$	&$.76$	\\
$4\times10^{10}$		&$-.73$	&$.80$		&$4\times10^{12}$	&$-.73$	&$.73$	\\
$8\times10^{10}$		&$-.80$	&$.86$		&$8\times10^{12}$	&$-.80$	&$.76$	\\
$16\times10^{10}$	&$-.88$	&$.68$		&$16\times10^{12}$	&$-.80$	&$.68$	\\
$32\times10^{10}$	&$-.88$	&$.78$		&$32\times10^{12}$	&$-.67$	&$.93$	\\
$64\times10^{10}$	&$-.66$	&$.74$		&$64\times10^{12}$	&$-.78$	&$.77$	\\
\end{tabular}

\vspace*{.5cm}

\begin{tabular}{r|r|r||r|r|r}
$x$	&$ M_\psi^-(x)$		&$ M_\psi^+(x)$		&$x$ 	&$ M_\psi^-(x)$	&	$ M_\psi^+(x)$	\\
\hline
$10^{14}$					&$-.79$	&$.72$		&$10^{16}$					&$-.88$	&$.74$		\\	
$2\times10^{14}$		&$-.60$	&$.76$		&$2\times10^{16}$		&$-.87$	&$.70$		\\
$4\times10^{14}$		&$-.65$	&$.73$		&$4\times10^{16}$		&$-.65$	&$.73$		\\
$8\times10^{14}$		&$-.81$	&$.88$		&$8\times10^{16}$		&$-.82$	&$.77$		\\
$16\times10^{14}$	&$-.66$	&$.86$		&$16\times10^{16}$		&$-.71$	&$.92$		\\
$32\times10^{14}$	&$-.74$	&$.86$		&$32\times10^{16}$		&$-.78$	&$.71$		\\
$64\times10^{14}$	&$-.73$	&$.66$		&$64\times10^{16}$		&$-.94$	&$.82$		\\
									&				&				&$128\times 10^{16}$	&$-.94$	&$.75$		\\
									&				&				&$256\times 10^{16}$ &$-.82$	&$.86$ 	\\
									&				&				&$512\times 10^{16}$	&$-.83$	&$.94$		\\
\end{tabular}
\end{center}
\end{table}

\subsection{Bounds for $\pi(x)$, $\pi^*(x)$, and $\vartheta(x)$}
We provide several elementary lemmas for deriving the bounds in Theorem \ref{t:explicit-bounds} from the calculated bounds for $\psi(x)$.
\begin{lemma}\label{l:psi->theta}
Let $1<a<b$ and suppose
\begin{equation}\label{e:R_psi-bounds}
c \leq  \frac{x-\psi(x)}{\sqrt x} \leq C
\end{equation}
holds for $x\in[a,b]$. Then
\begin{equation}\label{e:R_theta-upper}
\frac{x-\vartheta(x)}{\sqrt x} \leq C + 1 - c\,x^{-1/4} + 1.03883 \, \frac{x^{1/3} + x^{1/5} + 2\log(x)\,x^{1/13}}{\sqrt x}
\end{equation}
and
\begin{equation}\label{e:R_theta-lower}
\frac{x-\vartheta(x)}{\sqrt x} \geq c + 1 - C \, x^{-1/4}
\end{equation}
hold for $x\in[a^2,b]$.
\end{lemma}

\begin{proof}
We need to bound $\vartheta(x)$ in terms of $\psi(t)$. To this end we use
\begin{equation}\label{e:moebius}
\vartheta(x) = \sum_{k=1}^\infty \mu(k) \psi(x^{1/k}) = \sum_{k=1}^{\floor{2\log x}} \mu(k) \psi(x^{1/k}),
\end{equation}
and the bounds
\begin{align}
\psi(x) &\leq x \log x &\text{for $x\geq 1$},\label{e:psi-upper-trivial} \\
\psi(x) &< 1.03883 \,x &\text{for $x> 0$},\label{e:psi-upper-cheb} \\
\intertext{and}
\psi(x) &\geq 0.82 \,x &\text{for $x\geq 100$}.\label{e:psi-lower}
\end{align}
The first bound is trivial, the second is proven in \cite[Theorem 12]{RS62} and the third bound follows from \cite[Theorem 10]{RS62}. Now, since $\sum_{k=4}^n\mu(k) \leq 0$ for $n< 39$ and since $\psi(x^{1/k})$ decreases monotonously with increasing $k$, we get
\begin{align*}
\vartheta(x) &\leq \psi(x) -  \psi(\sqrt x) - \psi(x^{1/3}) + \sum_{n=39}^{\floor{2\log x}} \psi(x^{1/n}) \\
	&\leq \psi(x) - \psi(\sqrt x) - 0.82\, x^{1/3}  + \frac{2}{39}\, x^{1/39} \log(x)^2
\end{align*}
from \eqref{e:moebius} for $x\geq 10^6$, where we used \eqref{e:psi-upper-trivial} and \eqref{e:psi-lower} on the second line. The term $- 0.82\, x^{1/3}  + \frac{2}{39} \, x^{1/39} \log(x)^2$ is easily seen to be negative for $x\geq 10^6$, so we get
\begin{equation}\label{e:aux-theta-upper}
\vartheta(x) \leq \psi(x) - \psi(\sqrt x)
\end{equation}
first for $x\geq 10^6$, and then by directly checking the remaining values even for $x\geq 0$. For the lower bound we proceed in a similar way, using $\sum_{k=6}^n \mu(k) \geq 0$ for $n< 13$, which gives
\begin{align}
\vartheta(x) &\geq \psi(x) - \psi(\sqrt x) - \psi(x^{1/3}) - \psi(x^{1/5}) - \sum_{n=13}^{\floor{2\log x}} \psi(x^{1/n}) \notag\\
	&\geq \psi(x) - \psi(\sqrt x) - 1.03883 \,(x^{1/3} + x^{1/5} + 2\log(x) x^{1/13})\label{e:aux-theta-lower}
\end{align}
for $x\geq 1$, where we used \eqref{e:psi-upper-cheb} on the second line. Putting $\vartheta(x) = \psi(x) - \psi(\sqrt x) + r(x)$, the inequalities \eqref{e:R_theta-lower} and \eqref{e:R_theta-upper} now easily follow by inserting \eqref{e:aux-theta-upper}, respectively \eqref{e:aux-theta-lower} and \eqref{e:R_psi-bounds} into
\begin{equation*}
\frac{x - \vartheta(x)}{x} = \frac{x- \psi(x)}{x} + 1 - x^{-1/4} \frac{\sqrt x- \psi(\sqrt x)}{x^{1/4}}  - \frac{r(x)}{\sqrt x}.
\end{equation*}

\end{proof}

In order to prove \eqref{e:theta-upper} and \eqref{e:theta-lower} we first apply Lemma \ref{l:psi->theta} with $a=100$, $b=5\times 10^{10}$ and $-c=C=0.81$, which gives \eqref{e:theta-upper} and \eqref{e:theta-lower} for $10^7\leq x\leq 5\times 10^{10}$. Switching the parameters to $b=32\times 10^{12}$ and $-c = C = 0.88$ extends them to $5\times 10^{8} \leq x \leq 32\times 10^{12}$ and taking $b=10^{19}$ and $-c = C = 0.94$ gives them for $32\times 10^{12} \leq x \leq 10^{19}$. For the remaining values smaller than $10^7$ the bounds have been verified by a direct computation.

\begin{lemma}\label{l:psi->pi-star}
Let $b>10^7$, $12<a<b$, let $c<0$ and $C>0$ satisfy 
\begin{equation}
c \leq  \frac{x-\psi(x)}{\sqrt x} \leq C
\end{equation}
for all $x\in[a,b]$, and let
\[
A = \pi^*(a)-\li(a) + \frac{a-\psi(a)}{\log a}.
\]
Then we have
\[
\frac{\li(x) - \pi^*(x)}{\sqrt x / \log x} \leq  \frac{x-\psi(x)}{\sqrt x} + \frac{2 C}{\log x}\Bigl(1 + \frac{5}{\log  x}\Bigr) + A \frac{\log  x}{\sqrt x} ,
\]
and
\[
\frac{\li(x) - \pi^*(x)}{\sqrt x / \log x} \geq  \frac{x-\psi(x)}{\sqrt x} + \frac{2 c}{\log  x}\Bigl(1 + \frac{5}{\log x} \Bigr) + A \frac{\log  x}{\sqrt x}
\]
for all $x\in[\max\{a,10^7\},b]$.
\end{lemma}

\begin{proof}
Partial summation gives
\begin{equation}\label{e:psi->pi-star}
\pi^*(x) - \pi^*(a) = \li(x)-\li(a) - \frac{x-\psi(x)}{\log x} + \frac{a-\psi(a)}{\log a} - \int_{a}^x \frac{t-\psi(t)}{t \log(t)^2}\, dt.
\end{equation}
It thus suffices to prove
\begin{equation}\label{e:rem-integral}
0 \leq   \int_{a}^x \frac{dt}{\sqrt t \log(t)^2} \leq 2\frac{\sqrt x}{\log (x)^2} \Bigl( 1 + \frac{5}{\log x}\Bigr)
\end{equation}
for $a\geq12$ and $x\geq 10^7$. Applying the substitution $u=\log t$ gives
\begin{align*}
\int_{a}^x \frac{dt}{\sqrt t \log(t)^2} 
	&= \int_{\log a}^{\log x} \frac{e^{u/2}}{u^2}\, du	\\
	&=\int_{\log a}^{\log a + i\infty} \frac{e^{u/2}}{u^2}\, du - \int_{\log x}^{\log x + i\infty} \frac{e^{u/2}}{u^2}\, du.
\end{align*}
For any $\alpha>0$ we get
\begin{equation}
e^{-\alpha/2} \int_{\alpha}^{\alpha + i\infty} \frac{e^{u/2}}{u^2}\, du = -\frac{2}{\alpha^2} - \frac{8}{\alpha^3} + 24i\int_0^\infty \frac{e^{it/2}}{(\alpha + it)^4}\, dt
\end{equation}
by repeated integration by parts. Here, the last integral on the right hand side is bounded in absolute value by
\[
\int_{0}^\alpha \frac{dt}{\alpha^4} + \int_{\alpha}^{\infty}\frac{dt}{t^4} = \frac{4}{3\alpha^3},
\]
 and we get
\[
\int_{\alpha}^{\alpha + i\infty} \frac{e^{u/2}}{u^2}\, du = -2\frac{e^{\alpha/2}}{\alpha^2}\Bigl( 1 + \frac{4}{\alpha} + \Theta\Bigl(\frac{16}{\alpha^2}\Bigr)\Bigr).
\]
Thus, the integral on the left hand side is negative for $\alpha\geq \log (12)$. Furthermore, we have $\frac{16}{\alpha}\leq 1$ for $\alpha \geq \log(10^7)$ so we get \eqref{e:rem-integral}.
\end{proof}

Choosing $a=100$ in Lemma \ref{l:psi->pi-star} and using the bounds from \eqref{e:psi-bounds-eratosthenes} and Table \ref{tb:psi-bounds} gives the bounds listed in Table \ref{tb:pis-bounds}. Similarly, one obtains the bound \eqref{e:pi-star-bound} for $x \geq 10^7$ and the remaining values can again be checked by a direct computation.

\begin{table}
\caption{Upper and lower bounds $M^\pm_{\pi^*}(x)$ for $(\li(t)-\pi^*(t))\frac{\log t}{\sqrt t}$ in $[x,2x]$.}\label{tb:pis-bounds}

\vspace*{.5cm}
\begin{center}
\begin{tabular}{r|r|r||r|r|r}
$x$	&$ M_{\pi^*}^-(x)$		&$ M_{\pi^*}^+(x)$		&$x$ 	&$ M_{\pi^*}^-(x)$	&	$ M_{\pi^*}^+(x)$	\\
\hline
$10^{10}$			&$-.87$	&$.95$		&$10^{12}$			&$-.88$	&$.89$		\\	
$2\times10^{10}$	&$-.84$	&$.73$		&$2\times10^{12}$	&$-.87$	&$.84$		\\
$4\times10^{10}$	&$-.82$	&$.89$		&$4\times10^{12}$	&$-.81$	&$.81$		\\
$8\times10^{10}$	&$-.89$	&$.95$		&$8\times10^{12}$	&$-.87$	&$.84$		\\
$16\times10^{10}$	&$-.97$	&$.76$		&$16\times10^{12}$	&$-.87$	&$.76$		\\
$32\times10^{10}$	&$-.96$	&$.86$		&$32\times10^{12}$	&$-.74$	&$1$		\\
$64\times10^{10}$	&$-.74$	&$.82$		&$64\times10^{12}$	&$-.85$	&$.84$		\\
\end{tabular}
\vspace*{.5cm}

\begin{tabular}{r|r|r||r|r|r}
$x$	&$ M_{\pi^*}^-(x)$		&$ M_{\pi^*}^+(x)$		&$x$ 	&$ M_{\pi^*}^-(x)$	&	$ M_{\pi^*}^+(x)$	\\
\hline
$10^{14}$			&$-.86$	&$.79$		&$10^{16}$				&$-.94$	&$.80$		\\	
$2\times10^{14}$	&$-.67$	&$.83$		&$2\times10^{16}$		&$-.93$	&$.76$		\\
$4\times10^{14}$	&$-.72$	&$.80$		&$4\times10^{16}$		&$-.71$	&$.79$		\\
$8\times10^{14}$	&$-.87$	&$.95$		&$8\times10^{16}$		&$-.88$	&$.83$		\\
$16\times10^{14}$	&$-.72$	&$.93$		&$16\times10^{16}$		&$-.77$	&$.98$		\\
$32\times10^{14}$	&$-.80$	&$.92$		&$32\times10^{16}$		&$-.84$	&$.77$		\\
$64\times10^{14}$	&$-.79$	&$.72$		&$64\times10^{16}$		&$-1$	&$.88$		\\
			&				&			&$128\times 10^{16}$ 	&$-1$	&$.80$		\\
			&				&			&$256\times 10^{16}$	&$-.87$	&$.91$		\\
			&				&			&$512\times 10^{16}$	&$-.88$	&$.99$		\\
\end{tabular}
\end{center}
\end{table}

\begin{lemma}\label{l:theta->pi}
Let $b>10^7$, $12<a<b$, $c\leq 0$ and $C\geq 0$ such that 
\begin{equation}
c \leq  \frac{x-\vartheta(x)}{\sqrt{x}} \leq C
\end{equation}
holds for all $x\in[a,b]$, and let
\[
A =\pi(a) - \li(a) + \frac{a-\vartheta(a)}{\log a}.
\]
Then we have 
\[
\frac{\li(x)-\pi(x)}{\sqrt x/\log x} \leq  \frac{x-\vartheta(x)}{\sqrt{x}} + \frac{2C}{\log x}\Bigl(1 + \frac{5}{\log  x}\Bigr) + A \frac{\log  x}{\sqrt x},
\]
and
\[
\frac{\li(x)-\pi(x)}{\sqrt x/\log x} \geq  \frac{x-\vartheta(x)}{\sqrt{x}} + \frac{2 c}{\log  x}\Bigl(1 + \frac{5}{\log x} \Bigr) + A \frac{\log  x}{\sqrt x}
\]
for all $x\in[\max\{a,10^7\},b]$. Furthermore, the implication
\[
 t-\vartheta(t)>0\; \text{for $2\leq t \leq T$}\quad \Rightarrow\quad \li(t)-\pi(t) >0 \;\text{for $2\leq t \leq T$}
\]
holds.
\end{lemma}

\begin{proof}
The first assertion follows from 
\begin{equation}\label{e:theta->pi}
\pi(x) - \pi(a) = \li(x)-\li(a) - \frac{x-\vartheta(x)}{\log x} + \frac{a-\vartheta(a)}{\log a} - \int_{a}^x \frac{t-\vartheta(t)}{t \log(t)^2}\, dt
\end{equation}
in the same way as in the proof of Lemma \ref{l:psi->pi-star}. The second part is well-known and follows, e.g., by taking $a=10$ in \eqref{e:theta->pi} since
\[
\pi(10) - \li(10) + \frac{10-\vartheta(10)}{\log(10)} > 0.1.
\]
\end{proof}

Choosing $a=1,500$ in Lemma \ref{l:theta->pi} and using \eqref{e:theta-upper} gives \eqref{e:pi-upper} for $10^7\leq x \leq 10^{19}$ and the remaining values have again been checked directly. The bound \eqref{e:pi-lower} follows from \eqref{e:theta-lower} and \cite[Theorem 19]{RS62}.

\bibliographystyle{amsplain}
\bibliography{/home/janka/TeX/inputs/jankabib.bib}

\end{document}